\documentclass[10pt,twoside]{IEEEtran}
\usepackage {amssymb,rotating,graphicx,cite, calc, color, psfrag}
\usepackage{authblk}

\usepackage[cmex10]{amsmath}
\usepackage{amssymb,amsthm,amsfonts,mathrsfs}
\usepackage{graphicx}
\usepackage{enumitem,stackrel}
\usepackage{color}
\usepackage{mathtools,amssymb,bm}
\usepackage{amstext}
\usepackage{array}
\usepackage{algorithmicx}
\usepackage{hyperref}
\theoremstyle{remark}

\ifCLASSOPTIONcaptionsoff
  \usepackage[nomarkers]{endfloat}
 \let\MYoriglatexcaption\caption
 \renewcommand{\caption}[2][\relax]{\MYoriglatexcaption[#2]{#2}}
\fi
\newcommand{\RN}[1]{%
	\textup{\uppercase\expandafter{\romannumeral#1}}%
}

\newtheoremstyle{mystyle}
  {}
  {}
  {\itshape}
  {}
  {\bfseries}
  {.}
  { }
  {}
\theoremstyle{mystyle}

\newtheorem{theorem}{Theorem}{}
\newtheorem{proposition}{Proposition}{}
{}
{}
\newtheorem{defi}{Definition}{}
{}
\usepackage{epstopdf}

\begin{document}
%
\title{Riemannian Trust Region Method for Haplotype Assembly}

\author{Mohamad Mahdi~Mohades and Mohammad Hossein~Kahaei  \thanks{The authors are with the School of Electrical Engineering, Iran University of Science \& Technology, Tehran 16846-13114, Iran (e-mail: mohamad\_mohaddes@elec.iust.ac.ir; kahaei@iust.ac.ir).}}%

\maketitle

\begin{abstract}
In this letter we model the Haplotype assembly problem (HAP) as a maximization problem over an $(n-1)$-dimensional sphere. Due to nonconvexity of the feasible set, we propose a manifold optimization approach to solve the mentioned maximization problem. To escape local maxima as well as saddle points we utilize trust region method. Simulation results show that our proposed method is with high accuracy in estimation of Haplotype.
\end{abstract}


%
\IEEEpeerreviewmaketitle

\section{Introduction}\label{sec:Introduction}
\IEEEPARstart{H}{aplotype} is a string of single nucleotide polymorphisms (SNPs) of chromosomes \cite{Schwartz_2010}. Haplotypes are useful in studying human evolutionary history and drug discovery and development. Haplotypes can be assembled utilizing sequenced reads, where each read is a fragment of the two chromosomes \cite{bansal2008mcmc}. Such assembly can be performed by solving mathematical models of sequenced reads. For diploid organisms, each SNP site is considered to be either $+1$ or $-1$. This means that the result of haplotype assembly from the sequenced reads ought to be a vector of $\pm 1$ elements. Moreover, the obtained haplotype, say ${\bf{h}}_{n\times1}$, is corresponding to one of the two chromosomes and $-\bf{h}$ is corresponding to the other one.\\
One mathematical approach to find the haplotype is matrix completion. In this approach a read matrix, say ${\bf{M}}_{m\times n}$, containing sequenced reads is created. The elements of this matrix are either $\pm 1$  corresponding to the reads or $\times$ where there is no read. Then, it is required to replace the $\times$ elements of the matrix ${\bf{M}}$ with $+1$ or $-1$ so that the rank of the newly generated matrix, say $\overline{\bf{M}}$, be one. Then the factorization $\overline{\bf{M}}=\overline{\bf{c}}_{m\times1}\overline{\bf{h}}_{n\times1}^T$ gives the haplotype ${\bf{h}}_{n\times1}$ through applying sign function over $\overline{\bf{h}}_{n\times1}$.
When the reads are affected by noise, the sign of some of the entries of the matrix ${\bf{M}}_{m\times n}$ is changed. In this case using the following optimization problem the completed matrix $\overline{\bf{M}}$ is obtained. Please note that in HAP, whether to estimate ${\bf{h}}$ or $-{\bf{h}}$, the estimation is accurate.
\begin{equation} \label{Frobenius_Matrix_Completion}
\begin{array}{l}
\mathop {\min }\limits_{{\bf{X}}}\,\,\,\,\left\| {{{\rm{P}}_\Omega }\left({\bf{M}} \right) - {{\rm{P}}_\Omega }\left( \bf{X} \right)} \right\|_F^2\\
\,\,\,\,\,{\rm{subject \,\,to}}\,\,\,\,\,\, rank({\bf{X}})=1,
\end{array}
\end{equation}
where ${{\rm{P}}_\Omega }$ is the sampling operator acting as follows
\begin{equation} \label{P_Omega}
{{\rm{P}}_\Omega }({\bf{Q}})=\left\{ {\begin{array}{*{20}{c}}
{{{\rm{P}}_\Omega }({q_{ij})} = {q_{ij}} \,\,\,\,\,\,\left( {i,j} \right) \in \Omega }\\
{{{\rm{P}}_\Omega }({q_{ij})} =  0 \,\,\,\,\,\,\left( {i,j} \right) \notin \Omega }
\end{array}} \right.,
\end{equation}
and  $\Omega$ is the set of the observations.
There exist some approaches to solve the problem (\ref{Frobenius_Matrix_Completion}). Based on \cite{Bart}, the problem (\ref{Frobenius_Matrix_Completion}) can be directly solved over the manifold of rank-one matrices. Also, a variation of the formulation of the problem (\ref{Frobenius_Matrix_Completion}) has been solved in \cite{Keshavan_Few_Entries} over the Cartesian product of Grassmann manifolds. In fact, instead optimizing over the variable $\bf{X}$, the problem is solved over the variables $\bf{U}$, $\bf{\Sigma}$ and $\bf{V}$ where ${\bf{X}}={\bf{U}\bf{\Sigma}\bf{V}}^T$ is the singular value decomposition of $\bf{X}$.
Consider the following equivalent optimization problem for the problem (\ref{Frobenius_Matrix_Completion}).
\begin{equation} \label{Frobenius_Matrix_Completion_Equivalent}
\begin{array}{l}
\mathop {\min }\limits_{{\bf{X}}}\,\,\,\,\, rank({\bf{X}})\\
\,\,\,{\rm{subject \,\,to}}\,\,\,\,\,\, { \left\| {{{\rm{P}}_\Omega }\left({\bf{M}} \right) - {{\rm{P}}_\Omega }\left( \bf{X} \right)} \right\|_F^2}<\delta,
\end{array}
\end{equation}
where $\delta$ is a given value. In \cite{cai2010singular} the nonconvex objective function $rank({\bf{X}})$ has been replaced by the nuclear norm and the minimization problem turns into a convex one \cite{Candes_Recht}. Moreover, the problem (\ref{Frobenius_Matrix_Completion}) can be modeled as the following optimization problem.
\begin{equation} \label{Frobenius_factorization}
\mathop {\min }\limits_{{\bf{u}},{\bf{v}}} \,\,\left\| {{{\rm{P}}_\Omega }\left( {\bf{M}} \right) - {{\rm{P}}_\Omega }\left( {{\bf{u}}_{m\times1}{{\bf{v}}_{n\times1}^T}} \right)} \right\|_F^2,
\end{equation}
Alternating minimization is utilized to solve the above nonconvex problem \cite{Changxiao_Cai}.
In \cite{Mohades}, the authors have given an example to show that the proposed objective function of the problem (\ref{Frobenius_Matrix_Completion}) is not a reliable cost function in haplotype assembly. They instead have proposed an objective function to truly model the haplotype assembly problem. Using simulation results they have shown their approach is more accurate than previously proposed algorithms in haplotype assembly .

Apart from the aforementioned optimization approaches, it has been proposed to minimize the following minimum error correction (MEC) function to estimate the haplotype.

\begin{equation} \label{MEC_formulation}
{\rm{MEC}}\left( {{\bf{M}},{\bf{z}}} \right) = \sum\limits_{i = 1}^m {\min \left( {hd\left( {{{\bf{m}}_i},{\bf{z}}} \right),hd\left( {{{\bf{m}}_i}, - {\bf{z}}} \right)} \right)},
\end{equation}
where ${\bf{m}}_i$ is the $i$-th row of $\bf{M}$ and
\begin{equation} \label{Hamming_distance}
hd\left( {{{\bf{m}}_i},{\bf{z}}} \right) = \sum\limits_{\left. j \right|\left( {i,j} \right) \in \Omega }^{} {d\left( {{m_{ij}},{z_j}} \right)},
\end{equation}
where $d(\cdot,\cdot)$ equals $0$ for the same inputs and $1$ otherwise.
It is NP-Hard to obtain the optimal solution of the MEC problem (\ref{MEC_formulation}). Therefore, heuristic methods have been proposed to solve the problem (\ref{MEC_formulation}). For example, in \cite{HapCUT}, a heuristic combinatorial approach is proposed to find the haplotype. However, no provable guarantee for a good solution is proposed therein. Another heuristic approaches can be found in \cite{Puljiz_2016} and \cite{Hashemi}. 
In \cite{SDHAP}, the authors have firstly offered the following optimization problem to find the haplotype,
\begin{equation}\label{SDHAP_Formula}
\begin{array}{l}
\mathop {\max }\limits_{\bf{x}} \,\,\,\,\,{{\bf{x}}^T}{\bf{Wx}}\\
{\rm{subject\,\,to}}\,\,{\rm{ }}x_i^2 = 1
\end{array},
\end{equation}
where ${\bf{W}}$ is the adjacency matrix defined to evaluate the similarity of each pair of rows of the read matrix and $x_i$ is the $i$th entry of $\bf{x}$.
Then, they proposed a semidefinite program and solved it rapidly and accurately.\\
In this paper, we talk about some facts about the haplotype assembly problem (HAP). Then, based on such facts we propose a maximization problem to estimate haplotypes. 
Note that unlike the maximization problem \ref{SDHAP_Formula}, our formulation directly uses the reads and moreover is not a quadratic form. The proposed problem is defined over the $(n-1)$-dimensional sphere.
We use an algorithm to solve the mentioned problem and prove the convergence of the algorithm.
Simulation results illustrate that for  a large observation error probability, our method outweighs some of the previously proposed methods in the sense of haplotype estimation accuracy.

The remainder of the paper is organized as follows. In Section \ref{sec:Preliminary} some preliminaries are introduced to state our main result. Section \ref{Main_results} talks about the main result of the paper. Simulation results are presented in Section \ref{Simulation}. Conclusion is presented in Section \ref{section.conclusion}.

\section{Preliminaries}\label{sec:Preliminary}
In Section \ref{Main_results} we formulate HAP as a manifold optimization problem. To do so, we require some related concepts  as discussed in the following.

\begin{defi} \label{Tangent_space}\cite{Absil_book}
Let $\mathcal{M}$ be a manifold.  The set of all tangent vectors at point ${{\bf{x}}}\in{\mathcal{M}}$  is called the tangent space at point ${\bf{x}}$ and denoted by $T_{{\bf{x}}}{\mathcal{M}}$. Moreover, the disjoint collection of the tangent spaces is called tangent bundle and denoted by $T{\mathcal{M}}$.
\end{defi}

\begin{defi}\label{chart}
A given subset $\mathcal{U}$ of the manifold $\mathcal{M}$ along with a bijective mapping $\varphi$ between  $\mathcal{U}$  and an open subset of  $\mathbb{R}^d$ consist a pair  $\left(\mathcal{U},\varphi \right)$ which is called a $d$-dimensional chart of  $\mathcal{M}$.
\end{defi}

\begin{defi} \label{Diffeomorphism}\cite{Absil_book}
A differentiable bijective mapping between two manifolds is called diffeomorphism provided that its inverse is differentiable too.
\end{defi}

\begin{defi} \label{Parallel_Trans}\cite{Absil_book}
A tangent vector field $\xi$ on $\mathcal{M}$ is a smooth function which assigns to each point of the manifold a tangent vector belonging to tangent bundle. The gradient of a real valued smooth function $f$ over a manifold is an example of vector field which is denoted by ${\rm{grad}}f$. For example, tangent vector field $\xi$ over sphere $S^{n-1}$ is specified as:
\begin{equation}\label{gradient_Project}
\xi_{\bf{x}}={\rm{grad}}f\left( {\bf{x}} \right) = \left( {I - {\bf{x}}{{\bf{x}}^T}} \right){\rm{Grad}} f\left( {\bf{x}} \right), \forall {\bf{x}}\in S^{n-1},
\end {equation}
where ${\rm{Grad}} f\left( {\bf{x}} \right)$ is the gradient over the Euclidean space and $I$ is the identity matrix \cite{Absil_book}.

\end{defi}

\begin{defi}  \cite{Absil_book} \label{manifold_embed}
Let the manifold $\mathcal{N}$ be a subset of the manifold $\mathcal{M}$. When the manifold topology of $\mathcal{N}$ coincides the induced topology of $\mathcal{M}$, $\mathcal{N}$ is called embedded submanifold.
\end{defi}

\begin{defi}\cite{Absil_book}\label{Distance}
A manifold endowed  with a smoothly varying inner product, is called Riemannian manifold. The inner product over the manifold is denoted by either $g$ or $\left<\cdot , \cdot\right>$. Moreover, $\left<\cdot , \cdot\right>_{\bf{x}}$ illustrates the restriction of the inner product to the tangent space $T_{\bf{x}}\mathcal{M}$. Also, the metric induced by this norm is called Riemannian distance and denoted by $\rm{dist}(\cdot,\cdot)$.
\end{defi}

\begin{defi}\cite{Absil_book}\label{geodesic}
A locally distance minimizing curve  over a manifold is called geodesic. Specifically, straight lines over  Euclidean spaces are geodesics.
\end{defi}
\begin{defi}\cite{Absil_book}\label{minimizinggeodesic}
A globally distance minimizing curve over a manifold is called minimizing geodesic. For example, $\alpha$ is the minimizing geodesic between points ${\bf{x}}, {\bf{y}} \in S^{n-1}$ over sphere $S^{n-1}$ as:
\begin{equation}\label{Geodesic_Sphere}
\alpha \left( t \right) = \alpha \left( 0 \right)\cos \left( {\left\| {\dot \alpha \left( 0 \right)} \right\|\upsilon t} \right) + \dot \alpha \left( 0 \right)\frac{1}{{\upsilon\left\| {\dot \alpha \left( 0 \right)} \right\|}}\sin \left( {\left\| {\dot \alpha \left( 0 \right)} \right\|\upsilon t} \right),
\end{equation}
where ${\dot \alpha \left( 0 \right)}\in T_{\bf{x}}{S^{n-1}}$, $\alpha(0)={\bf{x}}$, and  $\alpha(1)={\bf{y}}$. For simplicity we consider $\upsilon=1$.
\end{defi}
\begin{defi}\cite{Absil_book}\label{Exponential_map}
Let $\mathcal{M}$ be a Riemannian manifold whose tangent bundle is $T\mathcal{M}$. Exponential map $\rm{Exp}$ is a mapping from $T\mathcal{M}$ to $\mathcal{M}$ so that for $v\in T\mathcal{M}$, $\rm{Exp(v)}$ is equal to $h$ at time $1$; where $h$ is the unique geodesic starting from the base point of $v$ with velocity $v$ at time $0$.
\end{defi}

\begin{defi} \label{injectivity_radius_defi} \cite{Absil_book}
Injectivity radius of the manifold $\mathcal{M}$ is defined as follows,
\begin {equation} \label{injectivity_radius_formul}
i\left( \mathcal{M} \right): = \mathop {\inf }\limits_{{\bf{x}} \in \mathcal{M}} \sup \left\{ {\varepsilon  > 0:{{\left. {{\rm{Exp}}_{\bf{x}}} \right|}_{{B_\varepsilon }\left( {{0_{\bf{x}}}} \right)}}\,{\rm{is}}\,\,{\rm{diffeomorphism}}} \right\},
\end {equation}
	where ${{\left. {{\rm{Exp}}_{\bf{x}}} \right|}_{{B_\varepsilon }\left( {{0_{\bf{x}}}} \right)}}$ shows the restriction of the mapping ${{\rm{Exp}}_{\bf{x}}}$ to the ball ${{B_\varepsilon }\left( {{0_{\bf{x}}}} \right)}$. Moreover, ${{B_\varepsilon }\left( {{0_{\bf{x}}}} \right)}$ is a normal neighborhood.
\end{defi}

\begin{defi} \label{Lipschitz_cont_diff} \cite{Absil_book}
Let $(\mathcal{M},g)$ own a positive injectivity radius of $i(\mathcal{M})$. Then the real valued function $f$ on $\mathcal{M}$ is $\rm{Lipschitz\,continuously \,differentiable}$ ($L-C^1$) provided that
\begin {enumerate}
\item $f$ is differentiable,
\item $\forall {\bf{x}},{\bf{y}}\in\mathcal{M}$ with ${\rm{dist}}({\bf{x}},{\bf{y}})<i(\mathcal{M})$, there exists $\beta$ for which
\begin{equation} \label{LCD_Ineq}
\left\| {P_\alpha ^{0 \leftarrow 1}{\rm{grad}}f\left( {\bf{y}} \right) - {\rm{grad}}f\left( {\bf{x}} \right)} \right\| \le \beta {\rm{dist}}\left( {{\bf{x}},{\bf{y}}} \right),
\end{equation}
where $\alpha$  is the unique minimizing geodesic satisfying $\alpha(0)={\bf{x}}$ and $\alpha(1)={\bf{y}}$. Moreover, $P_\alpha ^{0 \leftarrow 1}$ is an isometry operator; which translates the tangent vector ${\rm{grad}}f\left( y \right)\in T_{\bf{y}}{\mathcal{M}} $ to the tangent space $T_{\bf{x}}{\mathcal{M}}$ making it possible to differentiate the tangent vectors ${\rm{grad}}f\left( {\bf{y}} \right)$ and  ${\rm{grad}}f\left( {\bf{x}} \right)$, and is called parallel translation. Also, ${\rm{dist}}({\bf{x}},{\bf{y}})$ is obtained by taking infimum over the length of all curves joining ${\bf{x}}$ to ${\bf{y}}$ (see formula (3-30) of \cite{Absil_book}). For example, for the compact embedded submanifold $S^{n-1}$ we have 
\begin{equation}\label{Dist}
{\text{dist}}\left( {{\bf{x}},{\bf{y}}} \right) = {\int_0^1 {\left\langle {\dot \alpha \left( t \right),\dot \alpha \left( t \right)} \right\rangle } ^{\frac{1}{2}}}dt,
\end{equation}
 where $\left\langle {a,b} \right\rangle  = {a^T}b$.

\end {enumerate}
\end{defi}

\begin{proposition}(Lemma 7.4.7 of \cite{Absil_book})\label{Taylor}
Consider ${{B_\varepsilon }\left( {{{\bf{x}}}} \right)}$ as a normal neighborhood of ${\bf{x}}\in\mathcal{M}$ and $\zeta$ as a continuously differentiable tangent vector field over $\mathcal{M}$. Also, let $\alpha$ be the unique minimizing geodesic with $\alpha(0)={\bf{x}}$, $\alpha(1)={\bf{y}}$, and $\dot{\alpha}(1)=\xi$. Then $\forall {\bf{y}}\in {{B_\varepsilon }\left( {{{\bf{x}}}} \right)}$:
\begin{equation}\label{Taylor_Formul}
P_\alpha ^{0 \leftarrow 1}{\zeta _{\bf{y}}} = {\zeta _{\bf{x}}} + {\nabla _\xi }\zeta  + \int_0^1 {\left( {P_\alpha ^{0 \leftarrow \tau }{\nabla _{\dot \alpha \left( \tau  \right)}}\zeta  - {\nabla _\xi }\zeta } \right)} {\text{d}}\tau,
\end{equation}
where $\nabla$ is Riemannian connection which generalizes the concept of directional derivative of a vector field. Specifically, $\nabla$ for Sphere at point ${\bf{x}}$ is as follows:
\begin{equation}\label{Connection_Sphere}
{\nabla _\xi }{\zeta _{\bf{x}}} = \left( {I - {\bf{x}}{{\bf{x}}^T}} \right){\text{D}}{\zeta _{\bf{x}}}\left( {\bf{x}} \right)\left[ \xi  \right],
\end{equation}
where ${\text{D}}{\zeta _{\bf{x}}}\left( {\bf{x}} \right)\left[ \xi  \right]$ is the conventional directional derivative on a Euclidean space.
\end{proposition}
\begin{defi}\cite{Absil_book}\label{Hessian}
The Riemannian Hessian of the real valued function $g$ on the Riemannian manifold $\mathcal{M}$ at ${\bf{x}}\in\mathcal{M}$ is defined as follows:
\begin{equation}
{\rm{Hess}}f({\bf{x}}):T_{\bf{x}}\mathcal{M}\to T_{\bf{x}}\mathcal{M}:\xi\to{\nabla _\xi }{\rm{grad}}f, \forall{\xi\in T_{\bf{x}}\mathcal{M}}.
\end{equation}
\end{defi}

To solve a manifold optimization problem, Riemannian Trust Region (RTR) method can be utilized. RTR methods utilize second order geometry of the cost function which lets them escape saddle points and obtain better results compared to first order line search methods \cite{Absil_book}. RTR algorithm to minimize given cost function $f$ over the manifold $\mathcal{M}$ is presented in Table 1.  Before presenting RTR algorithm, let us define the trust region subproblem.
\begin{defi} \label{RTR_Subproblem_Defi}\cite{Absil_book}
Let $f({\bf{x}})$ be a real valued cost function over the manifold $\mathcal{M}$, the following quadratic optimization problem is called trust region subproblem,
\begin{equation}\label{RTR_Subproblem}
\begin{array}{*{20}{l}}
\hspace{-.7cm}{\mathop {\min }\limits_{\eta  \in {T_{{{\bf{x}}}}}{\cal M}} {{\hat m}_{{{\bf{x}}}}}\left( \eta  \right) = f\left( {{{\bf{x}}}} \right) + \left\langle {{\rm{grad}}f\left( {{{\bf{x}}}} \right),\eta } \right\rangle_{{{\bf{x}}}}  + 0.5\left\langle {{H}\left[ \eta  \right],\eta } \right\rangle_{{{\bf{x}}}} }\\
{{\mkern 1mu} {\mkern 1mu} {\mkern 1mu} {\mkern 1mu} {\mkern 1mu} {\mkern 1mu} {\mkern 1mu} {\mkern 1mu} {\mkern 1mu} {\mkern 1mu} {\mkern 1mu} {\mkern 1mu} {\mkern 1mu} {\mkern 1mu} {\mkern 1mu} {\mkern 1mu} {\mkern 1mu} {\mkern 1mu} {\mkern 1mu} {\mkern 1mu} {\mkern 1mu} {\mkern 1mu} {\mkern 1mu} {\mkern 1mu} {\mkern 1mu} {\mkern 1mu} {\mkern 1mu} {\mkern 1mu} {\mkern 1mu} {\mkern 1mu} {\rm{s}}.{\rm{t}}.{\mkern 1mu} {\mkern 1mu} {\mkern 1mu} {\mkern 1mu} {\mkern 1mu} {\mkern 1mu} {\mkern 1mu} {{\left\langle {\eta ,\eta } \right\rangle }_{{{\bf{x}}}}} \le \Delta^2}
\end{array},
\end{equation}
where $H$ is some symmetric operator on ${T_{{{\bf{x}}}}}\mathcal{M}$.
\end{defi}
\begin{center} \label{RTR}
\hspace{-1.1cm} \begin{tabular}{|c|c|}
 \hline
\hspace{-1.2cm}\textbf{Algorithm 1}: Riemannian trust region algorithm\\
 \hline
\hspace{-.6cm}\textbf{Requirements}:Riemannian manifold $(\mathcal{M},g)$; Cost function $f$ \\on $\mathcal{M}$; initial point ${\bf{x}}_0 \in \mathcal{M}$, Retraction function $R$, Scalars \\ $\bar{\Delta}>0$, $\Delta_0\in(0,\bar{\Delta})$ and $\rho'\in[0,0.25)$. \\
\hspace{-6.6cm}{\bf{for}}  $k=0,1,2,...$  {\bf{do}}\\
\hspace{-3.8cm}\textbf{Step 1}: Find $\eta_k$ by solving Problem \ref{RTR_Subproblem}. \\
\hspace{-3.5cm}\textbf{Step 2}: Evaluate ${\rho _k} = \frac{{f\left( {{{\bf{x}}_k}} \right) - f\left( {{R_{{{\bf{x}}_k}}}\left( {{\eta _k}} \right)} \right)}}{{{{\hat m}_{{{\bf{x}}_k}}}\left( {{0_{{{\bf{x}}_k}}}} \right) - {{\hat m}_{{{\bf{x}}_k}}}\left( {{\eta _k}} \right)}}$.\\
\hspace{-3cm}\textbf{Step 3}: $\bf{If}$ $\rho_k<0.25$ $\bf{then}$ $\Delta_{k+1} =0.25\Delta_k$,\\
$\bf{else if}$ $\rho_k>0.75$ and $\|\eta_k\|=\Delta_k$ $\bf{then}$, $\Delta_{k+1}={\rm{min}}(2\Delta_k,\bar{\Delta})$ \\
$\bf{else}$ $\Delta_{k+1}=\Delta_k$, $\bf{end\,if}$.\\
\hspace{-.2cm}\textbf{Step 4}: If $\rho_k>\rho'$ then ${\bf{x}}_{k+1}=R_{\bf{x}}(\eta_k)$ else ${\bf{x}}_{k+1}={\bf{x}}_k$; $\bf{end\,for}$\\
\hline
\end{tabular}
\end{center}

\begin{proposition}  \cite{Absil_book} \label{Convergence_Alg1}
Let $\{{\bf{x}}_k\}$ be a sequence generated by Alg. 1. Also, the following conditions are satisfied, 
\begin{enumerate}
\item Mapping $f$ is $C^1$ and bounded below at the level set $\{{\bf{x}}\in{\mathcal{M}}:f({\bf{x}})\leq f({\bf{x}}_0)\}$,
\item Mapping $f$ is $L-C^1$,
\item There exist $\mu > 0$ and $\delta_{\mu}>0$ such that the retraction function $R:T\mathcal{M}\to\mathcal{M}$ (see Definition 4.1.1 of \cite{Absil_book}) satisfies
\begin{equation} \label{Retraction_Condition}
\left\| \xi  \right\| \ge \mu {\rm{dist}}\left( {{\bf{x}},{R_{\bf{x}}}\xi } \right)\,\,\forall {\bf{x}} \in \mathcal{M},\forall \xi  \in {T_{\bf{x}}} \mathcal{M},\left\| \xi  \right\| \le {\delta _\mu },
\end {equation}

\item Mapping  $\widehat{f}:T\mathcal{M}\to{\mathbb{R}}:\xi\to f(R\xi)$ is radially $L-C^1$, i.e.,
$\exists {\beta _{RL}},{\delta _{RL}} > 0$ such that
\begin {equation}\left| {\frac{d}{{d\tau }}{{\hat f}_{\bf{x}}}\left( {\tau \xi } \right){|_{\tau  = t}} - \frac{d}{{d\tau }}{{\hat f}_{\bf{x}}}\left( {\tau \xi } \right){|_{\tau  = 0}}} \right| \le {\beta _{RL}}t,
\end{equation}
for all ${\bf{x}}\in\mathcal{M}$, $t<\delta_{RL}$ and $\xi \in T_{\bf{x}}\mathcal{M}$ with $\|\xi\|=1$.
\item There is a constant $\beta_H$ such that $\|H_k\|\leq\beta_H$ for all $k$, where $H_k$ is the symmetric operator defined in Definition \ref{RTR_Subproblem} at iteration $k$.\\
\item Any $\eta_k$ obtained in Step 1 of Alg. 1 satisfies inequality
\begin{equation} \label{Cauchy}
\begin{array}{l}
{{\hat m}_{{{\bf{x}}_k}}}\left( 0 \right) - {{\hat m}_{{{\bf{x}}_k}}}\left( {{\eta _k}} \right) \ge \\
\,\,\,\,\,\,\,\,\,\,\,\,\,{c_1}\left\| {{\rm{grad}}f\left( {{{\bf{x}}_k}} \right)} \right\|\min \left( {{\Delta _k},\frac{{\left\| {{\rm{grad}}f\left( {{{\bf{x}}_k}} \right)} \right\|}}{{\left\| {{H_k}} \right\|}}} \right)
\end{array},
\end{equation}
for some constant $c_1>0$, where ${{\left\| {{H_k}} \right\|}}$ is the operator norm of $ {{H_k}} $.
Then, the following holds:
\begin{equation}
\mathop {\lim }\limits_{k \to \infty } {\rm{grad}}f\left( {{{\bf{x}}_k}} \right) = 0.
\end{equation}

\end{enumerate}

\end{proposition}

\section {Main result} \label{Main_results}
In this section we fistly present some facts in regard of HAP. Then, we propose an optimization problem to estimate haplotypes. Finally, we discuss the convergence of an algorithm for solving the proposed optimization problem.\\
Let $\bf{M}$ be the \underline{noiseless} read matrix and  $\overline{\bf{M}}$ be the completion of $\bf{M}$. Then, the following statements hold.
\begin{enumerate}
\item $\overline{\bf{M}}$ is a rank one matrix for which there exists the factorization $\overline{\bf{M}}=\overline{\bf{c}}_{m\times1}\overline{\bf{h}}_{n\times1}^T$.
\item There is no difference for HAP to estimate ${\bf{h}}$ as $sign(\overline{\bf{h}})$ or $sign(-\overline{\bf{h}})$.  In other words, $\bf{h}$ is equivalent to $-\bf{h}$. \item There exists a unique maximizer for the problem,
\begin{equation} \label{Example_Maximize}
\mathop {\max }\limits_{\scriptstyle{\bf{x}} \in {\mathbb{R}^n,}\hfill\atop
\scriptstyle{\left\| {\bf{x}} \right\|_2} \le 1\hfill} \,\,\,{\left\| {{\overline{\bf{M}}}{\bf{x}}} \right\|_1}
\end{equation} up to the equivalence of $\bf{h}$ and $-\bf{h}$.
\item It can be verified that (\ref{Example_Maximize}) is equivalent to the following optimization problem over Sphere $S^{n-1}$,
\begin{equation} \label{Example_Maximize_Equiv}
\mathop {\max }\limits_{\scriptstyle{\bf{x}} \in {S^{n-1}}\hfill} \,\,\,{\left\| {{\overline{\bf{M}}}{\bf{x}}} \right\|_1}.
\end{equation}
\end{enumerate}
We therefore propose the next optimization problem to estimate the haplotype for the noisy HAP,
\begin{equation} \label{Proposed_Maximize}
\mathop {\max }\limits_{\scriptstyle{\bf{x}} \in {S^{n-1}}\hfill} \,\,\,{\left\| {{\left({{\rm{P}}_\Omega }\left({\bf{M}} \right)\right)}{\bf{x}}} \right\|_1},
\end{equation}
Please note that,  even though objective function of Problem \ref{Example_Maximize} is convex, we intend to maximize the objective function. Therefore, the solution of the optimization problem cannot be trivially obtained through convex optimization approaches. Moreover, due to Alg. 1 is a descent algorithm, we rewrite Problem (\ref{Proposed_Maximize}) in the following form:
\begin{equation} \label{Proposed_Minimize}
\mathop {\min }\limits_{\scriptstyle{\bf{x}} \in {S^{n-1}}\hfill} \,\,\,-{\left\| {{\left({{\rm{P}}_\Omega }\left({\bf{M}} \right)\right)}{\bf{x}}} \right\|_1}.
\end{equation}

 Let us make the objective function of  (\ref{Proposed_Minimize}) differentiable to easily use smooth optimization approaches for finding the solution. For this purpose, we propose the subsequent problem,
\begin{equation} \label{Proposed_Maximize_Diff}
\mathop {\min }\limits_{{\bf{x}} \in {S^{n - 1}}} \,\,\,\,\,f({\bf{x}})=-\sum\limits_{i = 1}^m {{{\left( {{{\left( { {{{\bf{M}}_{{\Omega _i}}}{\bf{x}}}} \right)}^2} + \varepsilon } \right)}^{\frac{1}{2}}}},
\end{equation}
where ${{\bf{M}}_{{\Omega _i}}}$ is the $i$th row of the matrix ${{\rm{P}}_\Omega }\left({\bf{M}} \right)$ and $\varepsilon$ is a very small positive value. Please note that for $\varepsilon=0$, the objective function of (\ref{Proposed_Maximize_Diff}) turns into the objective function of (\ref{Proposed_Minimize}).
 
\begin{theorem} \label{our_Theorem}
For the optimization problem (\ref{Proposed_Maximize_Diff}), Conditions of Proposition \ref{Convergence_Alg1} are satisfied and consequently Alg. 1 converges.
\end {theorem}
\begin{proof}
We require to prove that the conditions of Proposition \ref{Convergence_Alg1} are satisfied for $f({\bf{x}})$ in problem (\ref{Proposed_Maximize_Diff}).\\
{\bf{Condition 1;}} It is easy to see that $f({\bf{x}})$ is continuous. Let $(\mathcal{U},\varphi_i)$ be a chart for $S^{n-1}$, then $f$ is differentiable over $S^{n-1}$ if $f\circ \varphi_i^{-1}$ is differentiable \cite{Absil_book}. Let ${\bf{x}}=(x_1,\cdots,x_n)\in{S}^n$, then, a straightforward choice for $\varphi_i$ is to choose $\varphi_i$ so that $\varphi_i({\bf{x}})=(x_1,\cdots,x_{i-1},x_{i+1},\cdots,x_n)$, when $x_i\neq0$. Accordingly, it is easy to see that $f\circ \varphi_i^{-1}$ is differentiable and subsequently $f$ is differentiable. Therefore, $f({\bf{x}})$ is $C^1$. Moreover, for ${\bf{x}}\in{S^{n-1}}$ it is obvious that $f({\bf{x}})$ is lower bounded.\\
{\bf{Condition 2;}}  Let $\alpha(t)$ be the geodesic defined in Definition (\ref{minimizinggeodesic}) with $\upsilon=1$, and distance function be as Equation (\ref{Dist}).
Now, let  $\xi(0)=\dot{\alpha}(0)={\rm{grad}}f({\bf{x}})$. Being aware of the fact that ${{\bf{x}}^T}\rho  = 0,\,\,\,\,\forall \rho  \in {T_{\bf{x}}}{S^{n - 1}}$, specifically for $\rho={\rm{grad}}f({\bf{x}})$, it is easy to verify that,
\begin{equation}\label{dist_exact}
\begin{gathered}
  {\text{dist}}\left( {{\bf{x}},{\bf{y}}} \right) = {\left\| {{\text{grad}}f\left( {\bf{x}} \right)} \right\|_2}.
\end{gathered}
\end{equation}
Then, by considering $\zeta$ as the gradient vector field in equality (\ref{Taylor_Formul}), we have:
\begin{equation}\label{Verifying_cond_2}
\begin{gathered}
  \left\| {P_\alpha ^{0 \leftarrow 1}{\zeta _{\bf{y}}} - {\zeta _{\bf{x}}}} \right\| = \left\| {\int_0^1 {P_\alpha ^{0 \leftarrow \tau }{\nabla _{\dot \alpha \left( \tau  \right)}}\zeta } {\text{d}}\tau } \right\| \hfill \\
  \,\,\,\,\,\,\,\,\,\,\,\,\,\,\,\,\,\,\,\,\,\,\,\,\,\,\,\,\,\,\, \leqslant {\int_0^1 {\left\| {{\nabla _{\dot \alpha \left( \tau  \right)}}\zeta } \right\|} _2}{\text{d}}\tau,  \hfill \\ 
\end{gathered}
\end{equation}
where we used the isometry property of ${P_\alpha ^{0 \leftarrow \tau }}$. Using Equation (\ref{Connection_Sphere}), we have:
\begin{equation}
\begin{gathered}
  {\nabla _{\dot \alpha \left( \tau  \right)}}\zeta  = \left( {I - \alpha \left( \tau  \right)\alpha {{\left( \tau  \right)}^T}} \right)\mathop {\lim }\limits_{t \to 0} \left( {\frac{{{\zeta _{\alpha \left( \tau  \right) + t\dot \alpha \left( \tau  \right)}} - {\zeta _{\alpha \left( \tau  \right)}}}}{t}} \right).\hfill \\
\end{gathered}
\end{equation}
By calculating the Euclidean gradient for the cost function of Problem (\ref{Proposed_Maximize_Diff}) as
\begin{equation}\label{Euclid_Grad_f}
  {\text{Grad}}f\left( {\alpha \left( \tau  \right)} \right) =  - \sum\limits_{i = 1}^m {\left( {{\mathbf{M}}_{{\Omega _i}}^T{{\mathbf{M}}_{{\Omega _i}}}\alpha \left( \tau  \right)} \right){{\left( {{{\left( {{{\mathbf{M}}_{{\Omega _i}}}\alpha \left( \tau  \right)} \right)}^2} + \varepsilon } \right)}^{ - \frac{1}{2}}}},
\end{equation}
it is not difficult to verify that 
\begin{equation}
\mathop {\lim }\limits_{t \to 0} \left( {\frac{{{\zeta _{\alpha \left( \tau  \right) + t\dot \alpha \left( \tau  \right)}} - {\zeta _{\alpha \left( \tau  \right)}}}}{t}} \right) = \sum\limits_{i = 1}^m {\left( {\frac{{ - {\mathbf{M}}_{{\Omega _i}}^T{{\mathbf{M}}_{{\Omega _i}}}\dot \alpha \left( \tau  \right)}}{{{{\left( {{{\left( {{{\mathbf{M}}_{{\Omega _i}}}\alpha \left( \tau  \right)} \right)}^2} + \varepsilon } \right)}^{{1 \mathord{\left/
 {\vphantom {1 2}} \right.
 \kern-\nulldelimiterspace} 2}}}}}} \right)}.
\end{equation}
Then, using Equality (\ref{gradient_Project}), Cauchy-Shwarz inequality and considering that ${\left\| {\dot \alpha \left( \tau \right)} \right\|_2} = {\left\| {{\text{grad}}f\left( {\bf{x}} \right)} \right\|_2}$ and ${\left\| {\left( {I - \alpha \left( \tau  \right)\alpha {{\left( \tau  \right)}^T}} \right)} \right\|_F} \leqslant n$, we will have:
\begin{equation}\label{Hessian_Verify}
\int_0^1 {{{\left\| {{\nabla _{\dot \alpha \left( \tau  \right)}}\zeta } \right\|}_2}} {\text{d}}\tau  \leqslant n{\sum\limits_{i = 1}^m {\left\| {\frac{{\left( {{\mathbf{M}}_{{\Omega _i}}^T{{\mathbf{M}}_{{\Omega _i}}}} \right)}}{{\sqrt \varepsilon  }}} \right\|} _F}{\left\| {{\text{grad}}f\left( {\bf{x}} \right)} \right\|_2}.
\end{equation}
The aforementioned argument along with Equality (\ref{dist_exact}) show that $f$ is $L-C^1$ with $\beta=n\sum\limits_{i = 1}^m {{{\left\| {\frac{{\left( {{\mathbf{M}}_{{\Omega _i}}^T{{\mathbf{M}}_{{\Omega _i}}}} \right)}}{{\sqrt \varepsilon  }}} \right\|}_F}}$.
\\
{\bf{Condition  3;}} Let $R_{\bf{x}}$ be the restriction of retraction function $R$ to the tangent space $T_{\bf{x}}S^{n-1}$ which can be defined as ${R_{\bf{x}}}(\xi ) = \frac{{{\bf{x}} + \xi }}{{{{\left\| {{\bf{x}} + \xi } \right\|}_2}}}, \forall\xi\in T_{\bf{x}}S^{n-1}$ \cite{Absil_book}. It is not difficult to verify that for a given $\xi\in T_{\bf{x}}S^{n-1}$ the geodesic $\alpha(t)$, defined in Formula (\ref{Geodesic_Sphere}), satisfies $\alpha(0)={\bf{x}}$, $\alpha(1)=R_{\bf{x}}(\xi)$ when $\dot \alpha \left( 0 \right) = \frac{\xi }{{{{\left\| \xi  \right\|}_2}}}{\cos ^{ - 1}}\left( {\left\| {{\bf{x}} + \xi } \right\|_2^{ - 1}} \right)$. Then, using Equation (\ref{Dist}) we have:
\begin{equation}
{\text{dist}}\left( {{\bf{x}},{R_{\bf{x}}}(\xi )} \right) = {\left\| {\dot \alpha \left( 0 \right)} \right\|_2} =  \left| {{{\cos }^{ - 1}}\left( {{{\left( {1 + \left\| \xi  \right\|_2^2} \right)}^{\frac{{ - 1}}{2}}}} \right)} \right|.
\end{equation}\\
Now, let us find $\delta_\mu$ and $\mu$ for sake of inequality (\ref{Retraction_Condition}). It is easy to numerically verify that for $\mu=1$ and any arbitrary $\delta_\mu\ge0$,  inequality (\ref{Retraction_Condition}) holds.
\\
{\bf{Condition 4;}} Let us evaluate radially $L-C^1$ property of $\widehat{f}$. We have:
\begin{equation}
\frac{d}{{d\tau }}{{\hat f}_{\bf{x}}}\left( {\tau \xi } \right) = \frac{{\partial \gamma }}{{\partial \tau }}\frac{\partial }{{\partial \gamma }}f\left( \gamma  \right),
\end{equation}
where $\gamma  = \frac{{\left( {{\bf{x}} + \tau \xi } \right)}}{{{{\left\| {{\bf{x}} + \tau \xi } \right\|}_2}}}$. Now, using the fact that ${\bf{x}}$ is orthogonal to $\xi$, $\|{\bf{x}}\|_2=\|\xi\|_2=1$ and some simple calculus, for any arbitrary $t>0$ we obtain that:
\begin {equation}
\begin{gathered}
  \left| {\frac{d}{{d\tau }}{{\hat f}_{\bf{x}}}\left( {\tau \xi } \right){|_{\tau  = t}} - \frac{d}{{d\tau }}{{\hat f}_{\bf{x}}}\left( {\tau \xi } \right){|_{\tau  = 0}}} \right| \leqslant  \hfill \\
  \,\,\,\,\,\,\,t\left( {\sum\limits_{i = 1}^m {\left\| {{{\mathbf{M}}_{{\Omega _i}}}} \right\|_2^3 + \left\| {{{\mathbf{M}}_{{\Omega _i}}}} \right\|_2^2} } \right) \hfill \\ 
\end{gathered}.
\end{equation}
Meaning that radially $L-C^1$ condition has been satisfied with ${\beta _{RL}} = t\left( {\sum\limits_{i = 1}^m {\left\| {{{\mathbf{M}}_{{\Omega _i}}}} \right\|_2^3 + \left\| {{{\mathbf{M}}_{{\Omega _i}}}} \right\|_2^2} } \right)$.
\\
{\bf{Condition 5;}}Let us consider the symmetric operator $H_k$ be Hessian of function $f$ at point ${\bf{x}}_k$. Then, based on Definition \ref{Hessian}, the fact that $\|H_k\|\le\|H_k\|_F$ and also an argument discussed in verifying Condition 2, we just require to prove the boundedness of $\|{\rm{grad}}f({\bf{x}}_k)\|$ (see Inequality (\ref{Hessian_Verify})). Based on Equation (\ref{gradient_Project}) and using the Eulidean gradient of function $f$, as defined in Equation (\ref{Euclid_Grad_f}), we will have:
\begin{equation}
\begin{gathered}
  {\left\| {{\text{grad}}f\left( {\bf{x}} \right)} \right\|_2} = {\left\| {\left( {I - {\bf{x}}{{\bf{x}}^T}} \right){\text{Grad}}f\left( {\bf{x}} \right)} \right\|_2} \hfill \\
  \,\,\,\,\,\,\,\,\,\,\,\,\,\,\,\,\,\,\,\,\,\,\,\,\,\, \leqslant \frac{n}{{\sqrt \varepsilon  }}\sum\limits_{i = 1}^m {\left\| {{{\mathbf{M}}_{{\Omega _i}}}} \right\|_2^2}  \hfill \\ 
\end{gathered}. 
\end{equation}
This results in
\begin{equation}
\begin{gathered}
\left\| {{H_k}} \right\| \leqslant {\left\| {{H_k}} \right\|_F} \leqslant {\beta _H} = \frac{{{n^2}}}{\varepsilon }{\left( {\sum\limits_{i = 1}^m {\left\| {{{\mathbf{M}}_{{\Omega _i}}}} \right\|_2^2} } \right)^2}
\end{gathered}. 
\end{equation}
\\
{\bf{Condition 6;}} It is shown in \cite{Absil_book} that using the truncated conjugate gradient method (see Alg. 11 of \cite{Absil_book}), Inequality (\ref{Cauchy}) is satisfied with $c_1=1/2$.\\
The proof is complete.
\end {proof}

\section {Simulation Results} \label{Simulation}

Hamming distance of the original haplotype and its estimation, denoted by $\rm{hd}$, is our criterion to evaluate the performance of different methods. Please note that in HAP, it does not matter that the original haplotype is either $\bf{h}$ or $-\bf{h}$. Accordingly, to obtain $\rm{hd}$ we calculate the Hamming distance of the estimated haplotype with  both $\bf{h}$ and $-\bf{h}$ and choose the minimum one. Simulations are performed using synthetic data created by random generation of bipolar vectors ${\bf{h}}_{n\times1}$ and ${\bf{c}}_{m\times1}$ to construct ${\bf{M}}={\bf{c}}_{m\times1}{\bf{h}}_{n\times1}^T$. The observation set $\Omega$ is randomly produced with the probability of observation $\rm{pd}$. Also, some observed samples are erroneous  after changing their original sign; we show the set of erroneous samples by $\Omega_E$, where $\Omega_E\in\Omega$. 
 Fig. \ref{hd_pd} is depicted  for $m=250$, $n=300$, and $0.25\leq $\rm{pd}$\leq 0.7$.  Moreover, ${{\left| \Omega_E  \right|} \mathord{\left/
 {\vphantom {{\left| \Omega  \right|} {\left| {{\Omega _E}} \right|}}} \right.
 \kern-\nulldelimiterspace} {\left| {{\Omega}} \right|}}$ is set to $0.35$, where $\left|\cdot\right|$ denotes the cardinality of a set.  As seen, our method outperforms the other approaches by generating lower $\rm{hd}$.
\begin{figure}[!ht]
\centering
\begin{minipage}[b]{0.4\textwidth}
    \includegraphics[width=\textwidth]{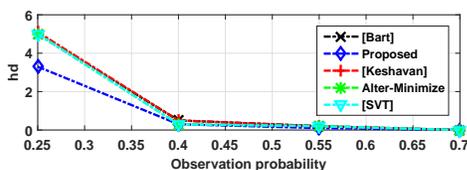}
  \caption{\label{hd_pd} Hamming distance.}
\end{minipage}
\hfill
\end{figure}

\section{Conclusion}\label{section.conclusion}
In this letter we proposed a new method for haplotype estimation. We properly modeled HAP over an (n-1)-dimensional Sphere. We  also discussed the convergence of a Riemannian trust region method. Simulation results confirmed our method outperforms some of the other haplotype assembly methods.



\bibliographystyle{ieeetr}
\bibliography{Hap_SPL_MMM_Manifold_97_08_05_Refrence}

\end{document}